\documentclass[12pt,leqno]{article}
\usepackage{amsfonts,amssymb,amsmath,amsthm,dutchcal}

\newcommand{\cu}{\mathcal{u}}

\DeclareMathOperator{\dt}{dt}
\DeclareMathOperator{\dx}{dx}
\DeclareMathOperator{\dist}{dist}

\newcommand{\D}{\overline{D}}
\newcommand{\bbR}{\mathbb{R}}

\DeclareMathOperator{\dive}{div}

\DeclareMathOperator{\op}{p}
\newcommand{\bop}{\boldsymbol{\op}}

\numberwithin{equation}{section}
\newtheorem{theorem}[equation]{Theorem}
\newtheorem{lemma}[equation]{Lemma}
\newtheorem{remark}[equation]{Remark}
\newtheorem{definition}[equation]{Definition}

\begin{document}
\begin{flushleft}

TITLE: A Hopf-type boundary point lemma for pairs of solutions to quasilinear equations.

\medskip

AUTHOR: Leobardo Rosales, Keimyung University

\medskip

ABSTRACT: We show a Hopf boundary point lemma for $u=u_{1}-u_{2},$ given $u_{1},u_{2} \in C^{1,\alpha}$ each weak solutions to a quasilinear equation $\sum_{i=1}^{n} D_{i}(A^{i}(x,u,Du))+B(x,u,Du)=0$ under mild boundedness assumptions on $A^{1},\ldots,A^{n},B.$ 

\medskip


\medskip


\section{Introduction}

In this work we give a Hopf-type boundary point result for pairs of solutions to certain \emph{quasilinear equations}. Our main theorem roughly is as follows:

\bigskip

{\bf Theorem \ref{main}:} \emph{Suppose $V \subset \bbR^{n}$ is a $C^{1,\alpha}$ open set for some $\alpha \in (0,1)$ with $0 \in \partial V,$ and suppose $A^{1},\ldots,A^{n} \in C^{2}(\overline{V} \times \bbR \times \bbR^{n})$ and $B \in C^{1}(\overline{V} \times \bbR \times \bbR^{n})$ where $A^{1},\ldots,A^{n}$ satisfy some mild boundedness assumptions. If $\cu_{1},\cu_{2} \in C^{1,\alpha}(\overline{V})$ are each weak solutions over $V$ to the equation
$$\sum_{i=1}^{n} D_{i}(A^{i}(x,\cu,D\cu)) + B(x,\cu,D\cu) = 0,$$
$\cu_{1}(0)=\cu_{2}(0)=0,$ and $\cu_{1}(x) \neq \cu_{2}(x)$ for all $x \in V,$ then $D\cu_{1}(0) \neq D\cu_{2}(0).$}

\bigskip

See as well Definition \ref{Qweaksolution}. The \emph{mild boundedness assumptions} are given by (ii),(iii),(iv) in the statement of Theorem \ref{main}.

\bigskip

Proving Theorem \ref{main} uses standard PDE techniques. For this, we show $\cu=\cu_{1}-\cu_{2}$ solves a \emph{linear equation} over $V$ of the form
\begin{equation} \label{intro}
\sum_{i,j=1}^{n} D_{i}(a^{ij}D_{j}\cu) + \sum_{i=1}^{n} c^{i} D_{i}\cu + d\cu = 0
\end{equation}
(see Definition \ref{weaksolution}) where $a^{ij} \in C^{0,\alpha}(\overline{V}),$ $c^{i} \in L^{\infty}(V),$ and $d \in L^{q}(V)$ for some $q>n.$ We then apply a generalized Hopf boundary point lemma to $\cu$ at the origin to conclude Theorem \ref{main}. More specifically, we apply the recent work \cite{R18b} of the author, given here for convenience to the reader as Lemma \ref{hopf}.

\bigskip

The main concern in order to apply Lemma \ref{hopf} is showing that the coefficient $d$ in \eqref{intro} is in $L^{q}(V)$ for some $q>n.$ This is not immediately evident, as $d$ is defined in terms of the second derivatives of $\cu_{1},\cu_{2},$ and while $\cu_{1},\cu_{2} \in C^{1,\alpha}(\overline{V})$ we can only conclude $\cu_{1},\cu_{2} \in C^{2}(V)$ by standard PDE arguments. Thus, $d \in L^{q}(V)$ for some $q>n$ must be carefully checked using the boundedness assumptions on $A^{1},\ldots,A^{n}$ (specifically, assumption (iv) of Theorem \ref{main}) and interior $C^{2}$ estimates for $\cu_{1},\cu_{2}.$ A natural question is whether we can circumvent this issue, by showing that $\cu=\cu_{1}-\cu_{2}$ solves a different linear equation (as in Definition \ref{weaksolution}) with coefficients which are not defined in terms of the second derivatives of $\cu_{1},\cu_{2}.$ However, this alternate strategy leads to much more cumbersome necessary assumptions on $A^{1},\ldots,A^{n},B;$ see Remark \ref{mainremark}(i). 

\bigskip

Before stating and proving our main result in \S 3, we give in \S 2 our basic definitions and results needed to prove Theorem \ref{main}. In particular, we state Lemma \ref{hopf} in \S 2, which is the version of the generalized Hopf boundary point lemma from \cite{R18b} which we need. 

\subsection{An application}

This work is a generalization of the argument used by the author in Lemma 4.1 of \cite{R18a} to study \emph{co-dimension one area-minimizing currents with tangentially immersed boundary}. 

\bigskip

We describe Lemma 4.1 of \cite{R18a} in a simple form. Suppose $V \subset \bbR^{n}$ is a $C^{1,\alpha}$ open set for some $\alpha \in (0,1)$ with $0 \in \partial V.$ Also suppose $\cu_{1},\cu_{2} \in C^{1,\alpha}(\overline{V}),$ $U \in C^{\infty}(\bbR^{n+1}),$ and for $\ell=1,2$ let
$$\Sigma_{\ell} = \{ (x,\cu_{\ell}(x),U(x,\cu_{\ell}(x))): x \in \overline{V} \} \subset \bbR^{n+2}.$$
Now let $\nu_{\ell}:\Sigma_{\ell} \rightarrow \bbR^{n+1}$ be the upward pointing unit normal of $\Sigma_{\ell}$ within the graph of $U;$ thus $\nu_{\ell}$ is tangent to the graph of $U,$ perpendicular to $\Sigma_{\ell},$ while $\nu_{\ell} \cdot e_{n}>0.$ Finally, we suppose there is a Lipschitz function $H:\bbR^{n+1} \rightarrow \bbR$ so that 
$$\int_{\Sigma_{\ell}} \dive_{\Sigma_{\ell}} X = \int_{\Sigma_{\ell}} X \cdot H \nu_{\ell}$$
for all smooth vector fields $X:\bbR^{n+1} \rightarrow \bbR^{n+1}$ with compact support in $V \times \bbR;$ in \cite{R18a} we say $\Sigma_{\ell}$ has \emph{Lipschitz co-oriented mean curvature} with respect to the graph of $U.$ In \cite{R18a} we argue, and Theorem \ref{main} implies, that $\cu_{1}(0)=\cu_{2}(0)$ while $\cu_{1}(x) \neq \cu_{2}(x)$ for all $x \in V$ implies $D\cu_{1}(0) \neq D\cu_{2}(0).$

\bigskip 

Most of the work in the proof of Lemma 4.1 of \cite{R18a} involves translating and rotating so that we are in a position to essentially apply Theorem \ref{main}. Also to this end, the calculations of the Appendix of \cite{R18a} are done essentially to verify that assumption (iv) of Theorem \ref{main} is satisfied. 

\subsection{Acknowledgements}

This work was partly conducted by the author while visiting the Korea Institute for Advanced Study, as an Associate Member.

\section{Preliminaries}

We will work in $\bbR^{n}$ with $n \geq 2.$ We denote the volume of the open unit ball $B_{1}(0) \subset \bbR^{n}$ by $\omega_{n} = \int_{B_{1}(0)} \dx.$ Standard notation for the various spaces of functions shall be used; we in particular note that $C^{1}_{c}(V;[0,\infty))$ shall denote the set of non-negative continuously differentiable functions with compact support in an open set $V \subseteq \bbR^{n}.$

\bigskip

Also, for $V \subseteq \bbR^{n}$ we shall write functions $A: \overline{V} \times \bbR \times \bbR^{n} \rightarrow \bbR$ by $A=A(x,z,p)$ where $x \in \overline{V},$ $z \in \bbR,$ and $p \in \bbR^{n}.$ For convenience to the reader, we shall let $D_{i}A$ denote the derivative of $A$ with respect to the $x_{i}$-variable for $i \in \{ 1,\ldots,n \},$ $\frac{\partial A}{\partial z}$ the derivative of $A$ with respect to the $z$-variable, and $\frac{\partial A}{\partial p_{j}}$ the derivative of $A$ with respect to the $p_{j}$-variable for $j \in \{ 1,\ldots,n \}.$

\bigskip

We begin by defining the quasilinear equations we shall consider.

\begin{definition} \label{Qweaksolution}
Let $V \subseteq \bbR^{n}$ be an open set, and suppose $A^{1},\ldots,A^{n},B \in C(V \times \bbR \times \bbR^{n}).$  We say $\cu \in C^{1}(V)$ \emph{is a weak solution over $V$ to the equation}
$$\sum_{i=1}^{n} D_{i} \left( A^{i}(x,\cu,D\cu) \right) + B(x,\cu,D\cu) = 0$$
if for all $\zeta \in C^{1}_{c}(V)$ we have
$$\int \sum_{i=1}^{n} A^{i}(x,\cu,D\cu) D_{i}\zeta - B(x,\cu,D\cu)\zeta \dx = 0.$$
\end{definition}

This is definition (13.2) of \cite{GT83}. We will also need to consider linear equations, in order to apply the results of \cite{R18b}. 

\begin{definition} \label{weaksolution}
Let $V \subset \bbR^{n}$ be an open set, and suppose $a^{ij},b^{i},c^{i} \in L^{2}(V)$ and $d \in L^{1}(V)$ for each $i,j \in \{ 1,\ldots,n\}.$ We say $\cu \in L^{\infty}(V) \cap W^{1,2}(V)$ \emph{is a weak solution over $V$ of the equation}
$$\sum_{i,j=1}^{n} D_{i} \left( a^{ij} D_{j}\cu + b^{i}\cu \right) + \sum_{i=1}^{n} c^{i} D_{i}\cu + d\cu \leq 0$$
(or more strictly, $=0$) if for all $\zeta \in C^{1}_{c}(V;[0,\infty))$ we have
$$\int \sum_{i,j=1}^{n} a^{ij} D_{j}\cu D_{i} \zeta + \sum_{i=1}^{n} \left( b^{i}\cu D_{i} \zeta - c^{i} (D_{i}\cu) \zeta \right) - d\cu \zeta \dx \geq 0$$
(respectively, $= 0$).
\end{definition}

The assumptions on the coefficients are merely to ensure integrability. We now introduce some terminology, to more conveniently state our results. 

\begin{definition} \label{uniformlyelliptic}
Let $V \subseteq \bbR^{n}.$
\begin{itemize}
 \item Suppose we have functions $a^{ij}: V \rightarrow \bbR$ for $i,j \in \{ 1,\ldots,n \}.$ We say $\{ a^{ij} \}_{i,j=1}^{n}$ are \emph{uniformly elliptic over $V$ with respect to $\lambda \in (0,\infty)$} if
$$\sum_{i,j=1}^{n} a^{ij}(x) \xi_{i} \xi_{j} \geq \lambda |\xi|^{2}$$
for each $x \in V$ and $\xi \in \bbR^{n}.$
 \item Suppose we have functions $A^{ij}: V \times \bbR \times \bbR^{n} \rightarrow \bbR$ for $i,j \in \{ 1,\ldots,n \}.$ We say $\{ A^{ij} \}_{i,j=1}^{n}$ are \emph{locally uniformly elliptic over $\overline{V} \times \bbR \times \bbR^{n}$} if for each $R \in (0,\infty)$ there is $\lambda_{R} \in (0,\infty)$ so that
$$\sum_{i,j=1}^{n} A^{ij}(x,z,p) \xi_{i} \xi_{j} \geq \lambda_{R} |\xi|^{2}$$
for each $(x,z,p) \in V \times [-R,R] \times \overline{B_{R}(0)}$ and $\xi \in \bbR^{n}.$
\end{itemize}
\end{definition}

Before we give the version of the generalized Hopf boundary point lemma from \cite{R18b} we shall need, we first give the following \emph{interior $C^{2}$ estimate}. We prove Lemma \ref{c2} using Theorem 8.32 of \cite{GT83}. The proof of Lemma \ref{c2} is standard, and known as the \emph{difference quotient method}.

\begin{lemma} \label{c2}
Suppose $V \subseteq \bbR^{n}$ is a bounded open set, and let $\alpha \in (0,1).$ Also suppose 
\begin{itemize}
 \item[(i)] $A^{i} \in C^{2}(\overline{V} \times \bbR \times \bbR^{n})$ for $i \in \{ 1,\ldots,n \}$ and $B \in C^{1}(\overline{V} \times \bbR \times \bbR^{n}),$
 \item[(ii)] $\left\{ \frac{\partial A^{i}}{\partial p_{j}} \right\}_{i,j=1}^{n}$ are locally uniformly elliptic over $\overline{V} \times \bbR \times \bbR^{n+1}.$
\end{itemize} 
If $\cu \in C^{1,\alpha}(\overline{V})$ is a weak solution over $V$ to the equation
\begin{equation} \label{c2_1}
\sum_{i=1}^{n} D_{i}(A^{i}(x,\cu,D\cu)) + B(x,\cu,D\cu) = 0,
\end{equation}
then $\cu \in C^{2}(V).$ Furthermore, with $R = \| \cu \|_{C^{1,\alpha}(\overline{V})}$ let $\lambda_{R} \in (0,\infty)$ be as in Definition \ref{uniformlyelliptic} applied to $\left\{ \frac{\partial A^{i}}{\partial p_{j}} \right\}_{i,j=1}^{n}.$ Then for each $x \in V$
$$|D^{2}\cu(x)| \leq \frac{C_{\ref{c2}}}{\min \{ 1,\dist(x,\partial V) \}}$$
where
$$C_{\ref{c2}}= C_{\ref{c2}}(n,\alpha,R,\lambda_{R},\{ \| A^{i} \|_{C^{2}(\overline{V} \times [-R,R] \times \overline{B_{R}(0)})} \}_{i=1}^{n},\| B \|_{C^{1}(\overline{V} \times [-R,R] \times \overline{B_{R}(0)})}).$$
\end{lemma}

\begin{proof}
Consider any $\hat{x} \in V$ and let $\rho = \dist(\hat{x},\partial V) \in (0,\infty).$ With fixed $h \in (-\frac{\rho}{2},\frac{\rho}{2})$ and $k \in \{ 1,\ldots,n \},$ define for $x \in \overline{B_{\frac{1}{2}}(0)}$
$$\begin{array}{lr}
\cu_{h,k}(x) = \frac{\cu(\rho x + \hat{x} +he_{k})-\cu(\rho x + \hat{x})}{h}, & \cu_{h,k} \in C^{1,\alpha}(\overline{B_{\frac{1}{2}}(0)}).
\end{array}$$
We wish to apply Theorem 8.32 of \cite{GT83} with $u=\cu_{h,k},$ $\Omega' = B_{\frac{1}{4}}(0),$ and $\Omega = B_{\frac{1}{2}}(0).$ We must thus compute that $\cu_{h,k}$ satisfies a linear equation as in Definition \ref{weaksolution} over $B_{\frac{1}{2}}(0).$

\bigskip

To do this, for any $\zeta \in C^{1}_{c}(B_{\frac{1}{2}}(0))$ we can input the test function
$$x \rightarrow \frac{\zeta(\frac{x-\hat{x}-he_{k}}{\rho})-\zeta(\frac{x-\hat{x}}{\rho})}{h} \text{ for } x \in B_{\rho}(\hat{x})$$
(after extending $\zeta$ to be zero outside of $B_{\frac{1}{2}}(0)$) into the weak equation \eqref{c2_1}. After a change of variables we conclude
$$\begin{aligned}
\frac{1}{\rho^{n+1}} \int \sum_{i=1}^{n} \frac{1}{h} \left( A^{i}(\mathbf{x},\cu(\mathbf{x},D\cu(\mathbf{x})) \Big|_{\mathbf{x}=\rho x + \hat{x}}^{\rho x + \hat{x} +he_{k}} \right) D_{i} \zeta \dx & \\
- \frac{1}{\rho^{n}} \int \frac{1}{h} \left( B(\mathbf{x},\cu(\mathbf{x},D\cu(\mathbf{x})) \Big|_{\mathbf{x}=\rho x + \hat{x}}^{\rho x + \hat{x} +he_{k}} \right) \zeta \dx & = 0. \\
\end{aligned}$$
Using single-variable calculus, we can compute that $\cu_{h,k}$ is a weak solution over $B_{\frac{1}{2}}(0)$ to the equation
\begin{equation} \label{c2_6}
\begin{aligned}
\sum_{i,j=1}^{n} D_{i} \left( a^{ij}_{h,k} D_{j}\cu_{h,k} + b^{i}_{h,k}\cu_{h,k} \right) + \sum_{i=1}^{n} c^{i}_{h,k} D_{i}\cu_{h,k} & + d_{h}\cu_{h,k} \\
= g_{h,k} & + \sum_{i=1}^{n} D_{i}f^{i}_{h,k}
\end{aligned}
\end{equation}
where we define for $x \in \overline{B_{\frac{1}{2}}(0)}$ and $i,j \in \{ 1,\ldots,n \}$
$$\begin{array}{ll}
a^{ij}_{h,k}(x) = \int_{0}^{1} \frac{\partial A^{i}}{\partial p_{j}}(P_{h,k}(t,x)) \dt, & b^{i}_{h,k}(x) = \rho \int_{0}^{1} \frac{\partial A^{i}}{\partial z}(P_{h,k}(t,x)) \dt, \\
c^{i}_{h,k}(x) = \rho \int_{0}^{1} \frac{\partial B}{\partial p_{i}}(P_{h,k}(t,x)) \dt, & d_{h,k}(x) = \rho^{2} \int_{0}^{1} \frac{\partial B}{\partial z}(P_{h,k}(t,x)) \dt, \\
g_{h,k}(x) = - \rho^{2} \int_{0}^{1} (D_{k}B)(P_{h,k}(t,x)) \dt, & f^{i}_{h,k}(x) = - \rho \int_{0}^{1} (D_{k}A^{i})(P_{h,k}(t,x)) \dt
\end{array}$$
with $P_{h,k}(t,x)$ for $t \in [0,1]$ and $h \in (-\frac{\rho}{2},\frac{\rho}{2})$ defined by
$$\begin{aligned}
P_{h,k}(t,x) = & t(\rho x + \hat{x}+he_{k},\cu(\rho x + \hat{x}+he_{k}),D\cu(\rho x + \hat{x}+he_{k})) \\
& + (1-t)(\rho x + \hat{x},\cu(\rho x + \hat{x}),D\cu(\rho x + \hat{x})).
\end{aligned}$$
Now let $L_{h,k}$ be the operator given by 
$$L_{h,k}u = \sum_{i,j=1}^{n} D_{i} \left( a^{ij}_{h,k} D_{j}u + b^{i}_{h,k} u \right) + \sum_{i=1}^{n} c^{i} D_{i}u + d_{h} u.$$
We now verify the hypothesis Theorem 8.32 of \cite{GT83} as follows:
\begin{itemize} 
 \item Let $R = \| \cu \|_{C^{1,\alpha}(\overline{V})},$ then $\{ a^{ij}_{h,k} \}_{i,j=1}^{n}$ are uniformly elliptic over $B_{\frac{1}{2}}(0)$ with respect to $\lambda_{R}$ by (ii).
 \item By (i), $\cu \in C^{1,\alpha}(\overline{V}),$ and $\rho = \dist(\hat{x},\partial V)$ we have
 $$\begin{array}{lr}
 a^{ij}_{h,k},b^{i}_{h,k},f_{h,k} \in C^{0,\alpha}(\overline{B_{\frac{1}{2}}(0)}), & c^{i}_{h,k},d_{h,k},g_{h,k} \in L^{\infty}(B_{\frac{1}{2}}(0))
 \end{array}$$
 for each $i,j \in \{ 1,\ldots,n \}.$ Furthermore, we have
 $$\max_{i,j=1,\ldots,n} \left\{ \begin{aligned}
& \| a^{ij}_{h,k} \|_{C^{0,\alpha}(\overline{B_{\frac{1}{2}}(0)})}, \| b^{i}_{h,k} \|_{C^{0,\alpha}(\overline{B_{\frac{1}{2}}(0)})}, \\
& \| c^{i}_{h,k} \|_{L^{\infty}(B_{\frac{1}{2}}(0))},\| d_{h,k} \|_{L^{\infty}(B_{\frac{1}{2}}(0))}
\end{aligned} \right\} \leq K_{\rho}$$
where we define, again with $R = \| \cu \|_{C^{1,\alpha}(\overline{V})},$
$$K_{\rho} = \max_{i=1,\ldots,n} \left\{ \begin{aligned}
& (1+\rho+\rho^{\alpha}R) \| A^{i} \|_{C^{2}(\overline{V} \times [-R,R] \times \overline{B_{R}(0)})}, \\
& \rho (1+\rho+\rho^{\alpha}R) \| A^{i} \|_{C^{2}(\overline{V} \times [-R,R] \times \overline{B_{R}(0)})}, \\
& \rho \| B \|_{C^{1}(\overline{V} \times [-R,R] \times \overline{B_{R}(0)})}, \\
& \rho^{2} \| B \|_{C^{1}(\overline{V} \times [-R,R] \times \overline{B_{R}(0)})}
\end{aligned} \right\}.$$
\end{itemize}
We conclude that the operator $L_{h,k}$ satisfies (8.5),(8.85) of \cite{GT83} with $\lambda=\lambda_{R}$ and $K=K_{\rho}.$ We can thus apply Theorem 8.32 of \cite{GT83}
$$\text{(with $a^{ij},b^{i},c^{i},d,g,f^{i}$ replaced respectively by $a^{ij}_{h,k},b^{i}_{h,k},c^{i}_{h,k},d_{h,k},g_{h,k},f^{i}_{h,k}$)}$$
over $\Omega = B_{\frac{1}{2}}(0)$ and with $\Omega'=B_{\frac{1}{4}}(0)$ (so that $d'=\dist(\Omega',\partial \Omega) = \frac{1}{4}$) to conclude (again with $R=\| \cu \|_{C^{1,\alpha}(\overline{V})}$)
\begin{equation} \label{c2_7}
\begin{aligned}
& \| \cu_{h,k} \|_{C^{1,\alpha}(\overline{B_{\frac{1}{4}}(0)})} \\
& \leq C \left( \| \cu_{h,k} \|_{L^{\infty}(B_{\frac{1}{2}}(0))} + \| g_{h,k} \|_{L^{\infty}(B_{\frac{1}{2}}(0))} + \sum_{i=1}^{n} \| f^{i}_{h,k} \|_{ C^{0,\alpha}(\overline{B_{\frac{1}{2}}(0)})} \right) \\
& \leq C \left( \begin{aligned}
R & + \rho^{2} \| B \|_{C^{1}(\overline{V} \times [-R,R] \times \overline{B_{R}(0)})} \\
& + \rho \left( 1 + \rho + \rho^{\alpha} R \right) \sum_{i=1}^{n} \| A^{i} \|_{C^{2}(\overline{V} \times [-R,R] \times \overline{B_{R}(0)})} 
\end{aligned} \right)
\end{aligned}
\end{equation}
where $C=C(n,\alpha,\lambda_{R},K_{\rho}) \in (0,\infty).$ In particular, the right-hand side is independent of $h \in (-\frac{\rho}{2},\frac{\rho}{2}).$ Letting $h \rightarrow 0,$ we can show using Arzela-Ascoli that $\cu \in C^{2}(B_{\frac{\rho}{2}}(\hat{x})).$

\bigskip

This shows $\cu \in C^{2}(V).$ We now prove the interior estimate for $D^{2}\cu.$ Again with $\hat{x} \in V,$ now set $\rho = \min \{ 1,\dist(\hat{x},\partial V) \}$ and repeat the above calculations. However, still with $R=\| \cu \|_{C^{1,\alpha}(\overline{V})},$ we replace $K_{\rho}$ with
$$K = \max_{i=1,\ldots,n} \left\{ (2+R) \| A^{i} \|_{C^{2}(\overline{V} \times [-R,R] \times \overline{B_{R}(0)})}, \| B \|_{C^{1}(\overline{V} \times [-R,R] \times \overline{B_{R}(0)})} \right\}.$$
Letting $h \rightarrow 0$ in \eqref{c2_7} we conclude
$$\begin{aligned}
& \rho |DD_{k}u(\hat{x})| = \lim_{h \rightarrow 0} |Du_{h,k}(0)| \\
& \leq C \lim_{h \rightarrow 0} \left( \| \cu_{h,k} \|_{L^{\infty}(B_{\frac{1}{2}}(0))} + \| g_{h,k} \|_{L^{\infty}(B_{\frac{1}{2}}(0))} + \sum_{i=1}^{n} \| f^{i}_{h,k} \|_{ C^{0,\alpha}(\overline{B_{\frac{1}{2}}(0)})} \right) \\
& \leq C \left( \begin{aligned}
R & + \| B \|_{C^{1}(\overline{V} \times [-R,R] \times \overline{B_{R}(0)})} \\
& + \rho \left( 2 + R \right) \sum_{i=1}^{n} \| A^{i} \|_{C^{2}(\overline{V} \times [-R,R] \times \overline{B_{R}(0)})} 
\end{aligned} \right)
\end{aligned}$$
where now $C=C(n,\alpha,\lambda_{R},K) \in (0,\infty).$
\end{proof}

\bigskip

We now state, for convenience, the version of the Hopf boundary point lemma from \cite{R18b} we shall need. To do so, we introduce some notation: let $B^{n-1}_{\rho}(0)$ denote the ball of radius $\rho \in (0,\infty)$ centered at the origin in $\bbR^{n-1}$; $\D$ shall denote differentiation over $\bbR^{n-1}.$ Also, we let $\bop: \bbR^{n} \rightarrow \bbR^{n-1}$ be the projection onto $\bbR^{n-1},$ and we will write points $y \in \bbR^{n-1}.$

\bigskip

Before we proceed, we note that the proof of Lemma \ref{hopf} refers to the \emph{Morrey space} $L^{1,\alpha}$; see Definition 2.1 of \cite{R18b}. Indeed, \cite{R18b} generalizes the Hopf boundary point lemma to linear equations (as in  Definition \ref{weaksolution}) with lower-order coefficient $d \in L^{1,\alpha}.$ Morrey spaces were introduced in \cite{M66} to study existence and regularity of solutions to elliptic systems, and since have been studied in and outside of partial differential equations. See for example \cite{M07}, which uses Morrey spaces to prove regularity results for solutions to non-linear divergence-form elliptic equations having inhomogeneous term consisting of a measure.

\begin{lemma} \label{hopf}
Let $\lambda \in (0,\infty)$ and $\alpha \in (0,1).$ Suppose $w \in C^{1,\alpha}(B^{n-1}_{1}(0);[0,\infty))$ satisfies $w(0)=0$ and $\D w(0)=0,$ and let 
$$W = \{ x \in B^{n-1}_{1}(0) \times (0,3): x_{n-1} > w(\bop(x)) \}.$$
Also suppose
\begin{itemize}
 \item[(i)] $a^{ij} \in C^{0,\alpha}(\overline{W}),$ $c^{i} \in L^{\infty}(W)$ for $i,j \in \{ 1,\ldots,n \},$ and $d \in L^{\frac{n}{1-\alpha}}(W),$
 \item[(ii)] $\{ a^{ij} \}_{i,j=1}^{n}$ are uniformly elliptic over $W$ with respect to $\lambda,$
 \item[(iii)] $d(x) \leq 0$ for each $x \in W,$
 \item[(iv)] $a^{ij}(0) = a^{ji}(0)$ for each $i,j \in \{ 1,\ldots,n \}.$
\end{itemize}
If $\cu \in C^{1}(\overline{W})$ is a weak solution over $W$ to the equation
\begin{equation} \label{hopf1}
\sum_{i,j=1}^{n} D_{i} \left( a^{ij}D_{j}\cu \right) + \sum_{i=1}^{n} c^{i}D_{i}\cu + d\cu \leq 0
\end{equation}
with $\cu(x)>\cu(0)=0$ for all $x \in W,$ then $D_{n}\cu(0) > 0.$
\end{lemma}

\begin{proof} 
Our goal is to apply the generalized Hopf boundary point lemma of \cite{R18b} to $\cu,$ after applying a change of variables. Choose $\rho \in (0,1)$ so that
\begin{equation} \label{hopf2}
\| w \|_{C^{1}(B^{n-1}_{\rho}(0))} < \max \left\{ 1,\sqrt{1+\frac{\lambda/2}{\sum_{i,j=1}^{n} \| a^{ij} \|_{C(\overline{W})}}}-1 \right\}.
\end{equation}
Define the map $\Psi_{\rho} \in C^{1,\alpha}(\overline{B_{1}(0)};\overline{W})$ by
$$\Psi_{\rho}(x) = \rho (x + e_{n}) + w(\bop(\rho x))e_{n} \text{ for } x \in B_{1}(0);$$
note that $\rho (x + e_{n}) + w(\bop(\rho x))e_{n} \in W$ for $x \in B_{1}(0).$ Now define
$$\begin{array}{lr}
\cu_{\rho}(x) = \cu(\Psi_{\rho}(x)) \text{ for } x \in \overline{B_{1}(0)}, & \cu_{\rho} \in C^{1}(\overline{B_{1}(0)}).
\end{array}$$
We derive a weak equation for $\cu_{\rho}$ over $B_{1}(0),$ by applying $\Psi_{\rho}$ as a change of variables to \eqref{hopf1}.

\bigskip

To this end, we compute for $x \in B_{1}(0)$
$$\begin{aligned}
D_{j}\cu_{\rho}(x) = & \rho (D_{j}\cu)(\Psi_{\rho}(x)) + \rho (D_{j}w)(\bop(\rho x)) (D_{n}\cu)(\Psi_{\rho}(x)) \\
= & \rho (D_{j}\cu)(\Psi_{\rho}(x)) + (D_{j}w)(\bop(\rho x)) D_{n}\cu_{\rho}(x) \\
& \text{for } j \in \{ 1,\ldots,n-1 \}, \\
D_{n}\cu_{\rho}(x) = & \rho (D_{n}\cu)(\Psi_{\rho}(x)).
\end{aligned}$$
Likewise, we compute for $\zeta \in C^{1}_{c}(B_{1}(0))$ and $x \in B_{1}(0)$
$$\begin{aligned}
D_{i}(\zeta(\Psi_{\rho}^{-1}(x))) = & \frac{1}{\rho} (D_{i}\zeta)(\Psi_{\rho}^{-1}(x)) - \frac{1}{\rho} (D_{i}w)(\bop(x)) (D_{n}\zeta)(\Psi_{\rho}^{-1}(x)) \\
= & \frac{1}{\rho} (D_{i}\zeta)(\Psi_{\rho}^{-1}(x)) - \frac{1}{\rho} (D_{i}w)(\bop(\rho \Psi_{\rho}^{-1}(x))) (D_{n}\zeta)(\Psi_{\rho}^{-1}(x)) \\
& \text{for } i \in \{ 1,\ldots,n-1 \}, \\
D_{n}(\zeta(\Psi_{\rho}^{-1}(x))) = & \frac{1}{\rho} (D_{n}\zeta)(\Psi_{\rho}^{-1}(x)).
\end{aligned}$$
These calculations, and using $\Psi_{\rho}:B_{1}(0) \rightarrow W$ as a change of variables in \eqref{hopf1}, imply $\cu_{\rho}$ is a weak solution over $B_{1}(0)$ to the equation
$$\sum_{i,j=1}^{n} D_{i} \left( a^{ij}_{\rho} D_{j}\cu_{\rho} \right) + \sum_{i=1}^{n} c^{i}_{\rho} D_{i} \cu_{\rho} + d_{\rho}\cu_{\rho} \leq 0$$
where we define $a^{ij}_{\rho}:\overline{B_{1}(0)} \rightarrow \bbR,$ $c^{i}_{\rho},d_{\rho}:B_{1}(0) \rightarrow \bbR$ for $i,j \in \{ 1,\ldots,n \}$ by
$$\begin{aligned}
a^{ij}_{\rho}(x) = & a^{ij}(\Psi_{\rho}(x)) \text{ for } i,j \in \{ 1,\ldots,n-1 \}, \\
a^{in}_{\rho}(x) = & a^{in}(\Psi_{\rho}(x)) - \sum_{\hat{j}=1}^{n-1} a^{i \hat{j}}(\Psi_{\rho}(x))(D_{\hat{j}}w)(\bop(\rho x)) \text{ for } i \in \{ 1,\ldots,n-1 \}, \\
a^{nj}_{\rho}(x) = & a^{nj}(\Psi_{\rho}(x)) - \sum_{\hat{i}=1}^{n-1} a^{\hat{i}j}(\Psi_{\rho}(x))(D_{\hat{i}}w)(\bop(\rho x)) \text{ for } j \in \{ 1,\ldots,n-1 \}, \\
a^{nn}_{\rho}(x) = & a^{nn}(\Psi_{\rho}(x)) + \sum_{\hat{i},\hat{j}=1}^{n-1} \left\{ \begin{aligned}
& a^{\hat{i} \hat{j}}(\Psi_{\rho}(x)) \\
& \times (D_{\hat{i}}w)(\bop(\rho x))(D_{\hat{j}}w)(\bop(\rho x)) 
\end{aligned} \right\}, & \\
c^{i}_{\rho}(x) = & \rho c^{i}(\Psi_{\rho}(x)) \text{ for } i \in \{ 1,\ldots,n-1 \}, \\
c^{n}_{\rho}(x) = & \rho c^{n}(\Psi_{\rho}(x)) - \rho \sum_{\hat{i}=1}^{n-1} (c^{\hat{i}}(\Psi_{\rho}(x))) (D_{\hat{i}}w)(\bop(\rho x)), & \\
d_{\rho}(z) = & \rho^{2} d(\Psi_{\rho}(x)). &
\end{aligned}$$
We now verify the hypothesis of Lemma 3.3 of \cite{R18b}:
\begin{itemize}
 \item $a^{ij}_{\rho} \in C^{0,\alpha}(\overline{B_{1}(0)}),$ $c^{i}_{\rho} \in L^{\infty}(B_{1}(0)) \subset L^{\frac{n}{1-\alpha}}(B_{1}(0))$ for $i,j \in \{ 1,\ldots,n-1 \},$ and 
$$d_{\rho} \in L^{\frac{n}{1-\alpha}}(B_{1}(0)) \subset L^{\frac{n}{2(1-\alpha)}}(B_{1}(0)) \cap L^{1,\alpha}(B_{1}(0))$$
by $\Psi_{\rho} \in C^{1,\alpha}(\overline{B_{1}(0)};\overline{W}),$ (i), Definition 2.1 of \cite{R18b}, and Remark 2.2 of \cite{R18b} with $q=\frac{n}{1-\alpha}.$
 \item $\{ a^{ij}_{\rho} \}_{i,j=1}^{n}$ are uniformly elliptic over $B_{1}(0)$ with respect to $\frac{\lambda}{2},$ by (ii) and \eqref{hopf2}.
 \item $\{ 0 \}_{i=1}^{n},d_{\rho}$ are weakly non-positive over $B_{1}(0)$ (see Definition 2.5 of \cite{R18b}) by  (iii).
 \item For each $i,j \in \{ 1,\ldots,n \}$
 $$a^{ij}_{\rho}(-e_{n})=a^{ij}(0)=a^{ji}(0) = a^{ji}_{\rho}(-e_{n})$$
by $w(0)=0$ and $\D w(0)=0$ (so that $\Psi_{\rho}(-e_{n})=0$).
\end{itemize}
Moreover, $w(0)=0$ implies
$$\cu_{\rho}(x) = \cu(\Psi_{\rho}(x)) > \cu(0) = \cu_{\rho}(-e_{n}) = 0.$$
We conclude by Theorem 4.1 of \cite{R18b} that $0 > D_{n}\cu_{\rho}(-e_{n}) = \rho D_{n}\cu(0).$
\end{proof}

\section{Main Theorem}

We are now ready to state and prove our main result.

\begin{theorem} \label{main} 
Suppose $\alpha \in (0,\frac{n-1}{2n-1}),$ and suppose $v \in C^{1,\alpha}(B^{n-1}_{1}(0))$ satisfies $v(0)=0,$ $\D v(0)=0,$ and $\| v \|_{C^{1,\alpha}(B^{n-1}_{1}(0))} \leq 1.$ With
$$V = \{ x \in B^{n-1}_{1}(0) \times (-3,3): x_{n} > v(\bop(x)) \},$$
suppose
\begin{itemize}
 \item[(i)] $A^{i} \in C^{2}(\overline{V} \times \bbR \times \bbR^{n})$ for each $i = 1,\ldots,n$ and $B \in C^{1}(\overline{V} \times \bbR \times \bbR^{n}),$
 \item[(ii)] $\left\{ \frac{\partial A^{i}}{\partial p_{j}} \right\}_{i,j=1}^{n}$ are locally uniformly elliptic over $\overline{V} \times \bbR \times \bbR^{n}$ \\ (see Definition \ref{uniformlyelliptic}),
 \item[(iii)] $\frac{\partial A^{i}}{\partial p_{j}}(0,0,p) = \frac{\partial A^{j}}{\partial p_{i}}(0,0,p)$ for each $i,j \in \{ 1,\ldots,n \}$ and $p \in \bbR^{n},$ 
 \item[(iv)] for each $R \in (0,\infty),$ there is $C_{R} \in (0,\infty)$ so that
 $$\sup_{(x,z,p) \in \overline{V} \times [-R,R] \times \overline{B_{R}(0)}} \left| \frac{\partial^{2}A^{i}}{\partial p_{j} \partial z}(x,z,p) \right| \leq C_{R}(|x|+|z|).$$
\end{itemize}
If $\cu_{1},\cu_{2} \in C^{1,\alpha}(\overline{V})$ are weak solutions over $V$ to the equation
\begin{equation} \label{main2}
\sum_{i,j=1}^{n} D_{i}(A^{i}(x,\cu,D\cu)) + B(x,\cu,D\cu) = 0
\end{equation}
with $\cu_{1}(0)=\cu_{2}(0)=0$ and $\cu_{1}(x) > \cu_{2}(x)$ for each $x \in V,$ then $D_{n}\cu_{1}(0) < D_{n}\cu_{2}(0).$
\end{theorem}

\begin{proof} Our goal is to apply Lemma \ref{hopf} to $\cu=\cu_{1}-\cu_{2}.$ 

\bigskip

First, we compute that $u$ solves a linear equation as in Definition \ref{weaksolution} over $V.$ Take any $\zeta \in C^{1}_{c}(V).$ Subtracting the weak equations \eqref{main2} for $\cu_{1},\cu_{2}$ we get
$$\int \left\{ \begin{aligned}
& \left( A^{i}(x,\cu_{1},D\cu_{1}) - A^{i}(x,\cu_{2},D\cu_{2}) \right) D_{i} \zeta \\
& - \Big( B(x,\cu_{1},D\cu_{1}) -B(x,\cu_{2},D\cu_{2}) \Big) \zeta 
\end{aligned} \right\} \dx = 0.$$
Using single-variable calculus, we can compute $\cu = \cu_{1}-\cu_{2} \in C^{1,\alpha}(\overline{V})$ is a weak solution over $V$ to the equation
$$\sum_{i,j=1}^{n} D_{i}(a^{ij} D_{j}\cu) + \sum_{i=1}^{n} c^{i} D_{i}\cu + d\cu = 0$$
where we define for $x \in \overline{V}$
\begin{equation} \label{main3}
\begin{aligned}
a^{ij}(x) = \int_{0}^{1} & \frac{\partial A^{i}}{\partial p_{j}}(P(t,x)) \dt \\
c^{i}(x) = \int_{0}^{1} & \frac{\partial B}{\partial p_{i}}(P(t,x)) + \frac{\partial A^{i}}{\partial z}(P(t,x)) \dt \\
d(x) = \int_{0}^{1} & \hspace{3in} \dt \\
& \overbrace{\begin{aligned}
& \frac{\partial B}{\partial z}(P(t,x)) \\
& + \sum_{i=1}^{n} \left\{ \begin{aligned}
& \frac{\partial D_{i} A^{i}}{\partial z}(P(t,x)) \\
& + \left( \begin{aligned} & \frac{\partial^{2} A^{i}}{\partial z^{2}}(P(t,x)) \\ & \times (t D_{i}\cu_{1}+(1-t)D_{i}\cu_{2}) \end{aligned} \right) \\
& + \sum_{j=1}^{n} \left( \begin{aligned} & \frac{\partial^{2}A^{i}}{\partial p_{j} \partial z}(P(t,x)) \\ & \times (tD_{i}D_{j}\cu_{1}+(1-t)D_{i}D_{j}\cu_{2}) \end{aligned} \right)
\end{aligned} \right. \end{aligned}}
\end{aligned}
\end{equation}
with as well for $t \in [0,1]$
$$P(t,x) = (x,t \cu_{1}(x)+ (1-t) \cu_{2}(x), t D\cu_{1}(x) + (1-t) D\cu_{2}(x)).$$
To see this more clearly, note that after using one-dimensional calculus, we further apply integration by parts to the term:
$$\begin{aligned}
\int \int_{0}^{1} & \frac{\partial A^{i}}{\partial z}(P(t,x)) \dt \cu D_{i}\zeta \dx \\
= & - \int D_{i} \left( \int_{0}^{1} \frac{\partial A^{i}}{\partial z}(P(t,x)) \dt \cu \right) \zeta \dx \\
= & - \int \int_{0}^{1} \frac{\partial A^{i}}{\partial z}(P(t,x)) \dt (D_{i} \cu) \zeta \dx \\
& - \int \int_{0}^{1} \frac{\partial D_{i}A^{i}}{\partial z}(P(t,x)) \dt \cu \zeta \dx \\
& - \int \int_{0}^{1} \left\{ \begin{aligned}
& \frac{\partial^{2} A^{i}}{\partial z^{2}}(P(t,x)) \\
& \times (tD_{i}\cu_{1}+(1-t)D_{i}\cu_{2})
\end{aligned} \right\} \dt \cu \zeta \dx \\
& - \int \int_{0}^{1} \sum_{j=1}^{n} \left\{ \begin{aligned}
& \frac{\partial^{2} A^{i}}{\partial p_{j} \partial z}(P(t,x)) \\
& \times (tD_{i}D_{j}\cu_{1}+(1-t)D_{i}D_{j}\cu_{2}) 
\end{aligned} \right\} \dt \cu \zeta \dx
\end{aligned}$$
using $A^{i} \in C^{2}(\overline{V} \times \bbR \times \bbR^{n})$ and $\cu_{1},\cu_{2} \in C^{2}(V)$ by Lemma \ref{c2}, which explains the definition of $c^{i},d;$ see Remark \ref{mainremark}(i).

\bigskip 

Moreover, note that for each $x \in V$ 
$$\cu(x)=\cu_{1}(x)-\cu_{2}(x) > \cu_{1}(0)-\cu_{2}(0)=0.$$
This implies that $\cu$ is a weak solution over $V$ of the equation
$$\sum_{i,j=1}^{n} D_{i}(a^{ij}D_{j}\cu) + \sum_{i=1}^{n-1} c^{i}D_{i}\cu + d_{-}\cu \leq 0$$
(see Definition \ref{weaksolution}) where $a^{ij},c^{i}$ for $i,j \in \{ 1,\ldots,n \}$ are as in \eqref{main3}, while 
\begin{equation} \label{main4}
d_{-}(x) = \max \{ 0,d(x) \} \text{ for } x \in V.
\end{equation}
As noted before, our aim is to apply Lemma \ref{hopf} to $\cu.$ However, we will not apply Lemma \ref{hopf} over the region $V,$ but instead over $W$ defined as follows.

\bigskip

Using $\| v \|_{C^{1,\alpha}(B^{n-1}_{1}(0))} \leq 1,$ define and compute
\begin{equation} \label{main5}
\begin{aligned}
w(y) = & 2v(y) + 3|y|^{1+\alpha} \text{ for } y \in B^{n-1}_{1}(0) \text{ where} \\
w \in & C^{1,\alpha}(B^{n-1}_{1}(0)), \\
w(0) = & 0 \text{ and } \D w(0)=0, \\
W = & \{ x \in B^{n-1}_{1}(0) \times (-3,3): x_{n} > w(\bop(x)) \}, \\
W \subseteq & V \cap \{ x \in B^{n-1}_{1}(0) \times (0,3): x_{n} > |\bop(x)|^{1+\alpha} \}, \\
\text{and } & \overline{B_{\frac{x_{n}}{4}}(x)} \subset V \text{ when } x \in W.
\end{aligned}
\end{equation}
Let us show the last claim. Fix $\hat{x} \in W,$ then the fifth item in \eqref{main5} implies $\hat{x}_{n} > 0.$ Moreover, for any $x \in \overline{B_{\frac{\hat{x}_{n}}{4}}(\hat{z})}$ we have by $\| v \|_{C^{1,\alpha}(B^{n-1}_{1}(0))} \leq 1$ and the definition of $w,W$ that
$$\begin{aligned}
x_{n} = & v(\bop(x)) + \hat{x}_{n} + x_{n} - \hat{x}_{n} + v(\bop(\hat{x})) - v(\bop(x)) - v(\bop(\hat{x})) \\
\geq & v(\bop(x)) + \hat{x}_{n} - |x_{n}-\hat{x}_{n}| - | \bop(\hat{x}) - \bop(x) | - v(\bop(\hat{x})) \\
\geq & v(\bop(x)) + \frac{\hat{x}_{n}}{2} - v(\bop(\hat{x})) \\
> & v(\bop(x)) + \frac{3}{2} |\bop(x)|^{1+\alpha} \geq v(\bop(x)).
\end{aligned}$$
Thus, $\overline{B_{\frac{\hat{x}_{n}}{4}}(\hat{z})} \subset V$ when $\hat{x} \in W.$

\bigskip 

We now check that $a^{ij},c^{i},d_{-}$ for $i,j \in \{ 1,\ldots,n \}$ as in \eqref{main3},\eqref{main4} satisfy the hypothesis of Lemma \ref{hopf} over $W,$ in reverse order:
\begin{itemize} 
 \item Using (iii) and $\cu_{1}(0)=\cu_{2}(0)=0$ we compute
$$\begin{aligned}
a^{ij}(0) = & \int_{0}^{1} \frac{\partial A^{i}}{\partial p_{j}}(0,0,tD\cu_{1}(0)+(1-t)D\cu_{2}(0)) \dt \\
= & \int_{0}^{1} \frac{\partial A^{j}}{\partial p_{i}}(0,0,tD\cu_{1}(0)+(1-t)D\cu_{2}(0)) \dt = a^{ji}(0)
\end{aligned}$$
for each $i,j \in \{ 1,\ldots,n \}.$
 \item $d_{-}(x) = \max \{ 0,d(x) \} \leq 0$ for each $x \in W.$
 \item By (ii) we have that $\{ a^{ij} \}_{i=1}^{n}$ are uniformly elliptic over $W$ with respect to $\lambda_{R} \in (0,\infty),$ where we set
 \begin{equation} \label{main6}
 R = \max \{ \| \cu_{1} \|_{C^{1,\alpha}(\overline{V})},\| \cu_{2} \|_{C^{1,\alpha}(\overline{V})} \}.
 \end{equation}
 \item By (i) and $\cu_{1},\cu_{2} \in C^{1,\alpha}(\overline{V})$ we immediately conclude
$$a^{ij} \in C^{0,\alpha}(\overline{W}) \text{ and } c^{i} \in L^{\infty}(W)$$
for each $i,j \in \{ 1,\ldots,n \}.$

\bigskip

We now show $d,$ and hence $d_{-},$ is in $L^{\frac{n}{1-\alpha}}(W).$ For this, since $0 \in \partial V$ we conclude by Lemma \ref{c2} that for each $\ell=1,2$
$$|D^{2}\cu(x)| \leq \frac{C_{\ref{c2}}}{\dist(x,\partial V)} \text{ for } x \in V \cap B_{1}(0)$$
where $C_{\ref{c2}}$ depends on
$$n, \ \alpha, \ R, \ \lambda_{R}, \ \{ \| A^{i} \|_{C^{2}(\overline{V} \times [-R,R] \times \overline{B_{R}(0)})} \}_{i=1}^{n}, \ \| B \|_{C^{1}(\overline{V} \times [-R,R] \times \overline{B_{R}(0)})}.$$
Now suppose $x \in W,$ then $\overline{B_{\frac{x_{n}}{4}}(x)} \subset V$ implies $\dist(x,\partial V) \geq \frac{x_{n}}{4}$ by \eqref{main5}. We thus conclude
\begin{equation} \label{main7}
|D^{2}\cu(x)| \leq \frac{4C_{\ref{c2}}}{x_{n}} \text{ for each } x \in W \cap B_{1}(0).
\end{equation}

\bigskip

We now consider each term in the definition of $d$ given in \eqref{main3}, which we bound independent of $t \in [0,1]$ over $W:$
\begin{itemize} 
 \item By (i) and \eqref{main6} we compute 
 $$\left| \frac{\partial B}{\partial z}(P(t,x)) \right| \leq \| B \|_{C^{1}(\overline{W} \times [-R,R] \times \overline{B_{R}(0)})}$$
 for $t \in [0,1]$ and $x \in W.$
 \item Similarly, we have for $t \in [0,1]$ and $x \in W$
 $$\left| \frac{\partial D_{i} A^{i}}{\partial z}(P(t,x)) \right| \leq \| A^{i} \|_{C^{2}(\overline{W} \times [-R,R] \times \overline{B_{R}(0)})}$$
and
$$\begin{aligned}
\Big| \frac{\partial^{2} A^{i}}{\partial z^{2}} (P(t,x)) \cdot & (t D_{i}\cu_{1}+(1-t)D_{i}\cu_{2}) \Big| \\
& \leq  \| A^{i} \|_{C^{2}(\overline{W} \times [-R,R] \times \overline{B_{R}(0)})} R
\end{aligned}$$
for each $i \in \{ 1,\ldots,n \}.$
 \item For $x \in W \cap B_{1}(0)$ we compute using \eqref{main7}, (iv) with \eqref{main6}, $\cu_{1}(0)=\cu_{2}(0)=0,$ and 
 $$W \subseteq \{ x \in B^{n-1}_{1}(0) \times (0,3): x_{n} > |\bop(x)|^{1+\alpha} \}$$
 by \eqref{main5} that when $\bop(x) \neq 0$
$$\begin{aligned}
\Big| \frac{\partial^{2}A^{i}}{\partial p_{j} \partial z}(P(t,x)) & \cdot (tD_{i}D_{j}\cu_{1}+(1-t)D_{i}D_{j}\cu_{2}) \Big| \\
& \leq \left| \frac{\partial^{2}A^{i}}{\partial p_{j} \partial z}(P(t,x)) \right| \frac{4C_{\ref{c2}}}{x_{n}} \\
& \leq \frac{4C_{\ref{c2}}C_{R}(|x|+|t\cu_{1}(x)+(1-t)\cu_{2}(x)|)}{x_{n}} \\
& \leq 4C_{\ref{c2}}C_{R} \left( \frac{|\bop(x)|}{x_{n}} + 1 + R \left( \frac{|\bop(x)|}{x_{n}} + 1 \right) \right) \\
& \leq 4C_{\ref{c2}}C_{R}(1+R) \left( \frac{|\bop(x)|}{|\bop(x)|^{1+\alpha}} + 1 \right) \\
& \leq \frac{8C_{\ref{c2}}C_{R}(1+R)}{|\bop(x)|^{\alpha}} 
\end{aligned}$$ 
for each $i,j \in \{ 1,\ldots,n \}.$ Moreover, note that we can compute using Tonelli's theorem
$$\int_{W \cap B_{1}(0)} \left( \frac{1}{|\bop(x)|^{\alpha}} \right)^{\frac{n}{1-\alpha}} \dx < \infty$$
since $\alpha \in (0,\frac{n-1}{2n-1}).$ 

\bigskip

Now suppose $x \in W \setminus B_{1}(0).$ In this case, either $x_{n} > \frac{1}{2}$ or $|\bop(x)| > \frac{1}{2},$ both of which imply by \eqref{main5} that $x_{n} > 2^{-1-\alpha}.$ We thus compute by \eqref{main7} and (iv) with \eqref{main6} that
$$\begin{aligned}
\Big| \frac{\partial^{2}A^{i}}{\partial p_{j} \partial z}(P(t,x)) & \cdot (tD_{i}D_{j}\cu_{1}+(1-t)D_{i}D_{j}\cu_{2}) \Big| \\
& \leq \| A^{i} \|_{C^{2}(\overline{V} \times [-R,R] \times \overline{B_{R}(0)})} \frac{4C_{\ref{c2}}}{x_{n}} \\
& \leq 2^{3+\alpha} \| A^{i} \|_{C^{2}(\overline{V} \times [-R,R] \times \overline{B_{R}(0)})} C_{\ref{c2}}
\end{aligned}$$
for each $i,j \in \{ 1,\ldots,n \}.$
\end{itemize}
We conclude $d,$ and hence $d_{-},$ is in $L^{\frac{n}{1-\alpha}}(W).$
\end{itemize}
Since $\cu(x) = \cu_{1}(x)-\cu_{2}(x)>\cu_{1}(0)-\cu_{2}(0) = 0$ for each $x \in W,$ we conclude by Lemma \ref{hopf} that $D_{n}\cu_{1}(0)<D_{n}\cu_{2}(0).$
\end{proof}

\begin{remark} \label{mainremark}
We remark on the proof and statement of Theorem \ref{main}.
\begin{itemize}
 \item[(i)] Observe that in the proof of Theorem \ref{main}, we could instead show $\cu = \cu_{1}-\cu_{2}$ solves an equation over $V$ of the form
$$\sum_{i,j=1}^{n} D_{i}(a^{ij} D_{j} \cu_{\ell} + b^{i} \cu_{\ell}) + \sum_{i=1}^{n} c^{i} D_{i} \cu_{\ell} + d \cu_{\ell} = 0,$$
where, as opposed to \eqref{main3}, the coefficients are only defined in terms of $\cu,D\cu$ and the first derivatives of $A^{1},\ldots,A^{n},B.$ The idea is to try to avoid setting hypothesis (iv).

\bigskip

However, applying the generalized Hopf boundary point lemma of \cite{R18b} in this case requires showing $\{ b^{i} \}_{i=1}^{n},d$ are weakly non-positive (see Definition 2.5 of \cite{R18b}). But in considering particular examples, this may be difficult (or impossible) to verify. Meanwhile, assumption (iv) of Theorem \ref{main} is far more accessible.
 \item[(ii)] Note that we need not specifically assume $\alpha \in (0,\frac{n-1}{2n-1}),$ since $\beta \in (0,\alpha)$ implies $C^{1,\alpha}(\Omega) \subset C^{1,\beta}(\Omega).$
 \item[(iii)] We need not assume $\| v \|_{C^{1,\alpha}(B^{n-1}_{1}(0))} \leq 1,$ by simply rescaling (as in the proofs of Lemmas \ref{c2},\ref{hopf}).
 \item[(iv)] What we really need in order to prove Theorem \ref{main} is
$$\Big| \frac{\partial^{2}A^{i}}{\partial p_{j} \partial z}(P(t,x)) \cdot (tD_{i}D_{j}\cu_{1}+(1-t)D_{i}D_{j}\cu_{2}) \Big| \leq \varphi(x)$$
 for each $t \in [0,1]$ and $x \in V,$ where $\varphi \in L^{q}(V)$ for some $q>n.$ Assumption (iv) merely guarantees this.
\end{itemize}
\end{remark}

\end{flushleft}
\end{document}